\documentclass[a4paper,10pt]{amsart}
\usepackage[utf8]{inputenc}
\usepackage{amsmath,amssymb,amsthm}
\usepackage{xcolor}
\usepackage{enumerate}

\usepackage{hyperref}

\theoremstyle{plain}

\newtheorem{thm}{Theorem}[section]
\newtheorem{corollary}[thm]{Corollary}
\newtheorem{lemma}[thm]{Lemma}

\newtheorem{conjecture}{Conjecture}

\theoremstyle{definition}

\newtheorem{definition}[thm]{Definition}

\newtheorem*{prob}{Problem}

\newcommand{\rad}{\operatorname{rad}}
\newcommand{\Li}{\operatorname{Li}}
\newcommand{\Hom}{\operatorname{Hom}}

\newcommand{\bbF}{\mathbb{F}}
\newcommand{\bbZ}{\mathbb{Z}}
\newcommand{\bbN}{\mathbb{N}}
\newcommand{\bbR}{\mathbb{R}}

\newcommand{\bfP}{\mathbf{P}}
\newcommand{\bfE}{\mathbf{E}}
\newcommand{\bfVar}{\mathrm{Var}}
\renewcommand{\Re}{\mathrm{Re}}
\renewcommand{\epsilon}{\varepsilon}

\numberwithin{equation}{section}

\title[Weil abscissa for procyclic groups]{On the abscissae of Weil representation zeta functions for procyclic groups}
\author{Steffen Kionke}
\address{FernUniversit\"at in Hagen, Fakult\"at f\"ur Mathematik und Informatik, 58084 Hagen, Germany}
\email{steffen.kionke@fernuni-hagen.de}


\date{\today}

\begin{document}
\renewcommand{\labelenumi}{(\roman{enumi})}
\renewcommand{\labelenumii}{(\alph{enumii})}

\begin{abstract}
A famous conjecture of Chowla on the least primes in arithmetic progressions implies that the abscissa of convergence of the Weil representation zeta function for a procyclic group $G$ only depends on the set $S$ of primes dividing the order of $G$ and that it agrees with the abscissa of the Dedekind zeta function of $\mathbb{Z}[p^{-1}\mid p \not\in S]$. 
Here we show that these consequences hold unconditionally for random procyclic groups in a suitable model. As a corollary, every real number $1 \leq \beta \leq 2$ is the Weil abscissa of some procyclic group.
\end{abstract}

\maketitle

\section{Introduction}
Dirichlet generating functions are a well-established tool in enumerative algebra, especially in asymptotic group theory; see for instance \cite{duSautoySegal, Voll1}. In particular, generating functions enumerating complex irreducible representations of infinite groups received considerable attention; see~\cite{AKOV,BudurZordan,KiKl,LL,StasinskiVoll} and references therein. In \cite{CCKV24} a generating function enumerating absolutely irreducible representations of a (profinite) group $G$ over all finite fields was defined that is reminiscent of the Hasse-Weil zeta function of an algebraic variety $V$; cf.~\cite{Weil}.
This \emph{Weil representation zeta function} $\zeta_G^W(s)$ of a profinite group $G$
converges absolutely for $\Re(s)$ sufficiently large, if $G$ has ``uniformly bounded exponential 
representation growth'' (UBERG); see \cite[Cor.~2.3]{CCKV24} and \cite{KV} for more on UBERG groups. The abscissa of convergence $\alpha(G)$ of $\zeta^W_G(s)$ -- called \emph{Weil abscissa} here -- is an intricate invariant. The Weil abscissa of free pronilpotent and free prosoluable groups was determined in~\cite[Thm.~C]{CCKV24} . 

\medskip

Let $G$ be a finitely generated profinite abelian group. Then $G$ has UBERG and -- since absolutely irreducible representations of $G$ are one-dimensional -- the Weil representation zeta function
is 
\begin{equation*}\label{eq:defn_zeta}
	\zeta^W_G(s) := \exp  \Bigl( \sum_{p \in \mathcal{P}}\sum_{j=1}^\infty |\Hom(G,\bbF_{p^j}^\times)| \frac{p^{-sj}}{j}\Bigr)
\end{equation*}
where $\mathcal{P}$ denotes the set of primes.
The Weil abscissa is known for some abelian groups, e.g.~$\alpha(\widehat{\bbZ}^r) = r+1$ and $\alpha(\bbZ_p) = 1$ for every prime $p$ \cite[Thm.~C]{CCKV24}.
The investigations for $G = \bbZ_p^r$ in \cite{CCKV24}  revealed that the 
Weil abscissa can be related to questions about primes in arithmetic progressions.

\medskip

For every non-empty set $S$ of primes, define the procyclic groups $H_S =  \prod_{p \in S} \bbZ_p$ and $G_S = \prod_{p\in S} C_p$ (here $C_p$ is cyclic of order $p$). Here we investigate the Weil abscissa of $H_S$ and $G_S$ inspired by the following conjecture.

\begin{conjecture}\label{conjA}
For every set $S$ of primes
$\alpha(G_S) = \alpha(H_S) = \alpha_D(S)+1$. 
\end{conjecture}

Here $\alpha_D(S)$ denotes the abscissa of convergence of the partial Riemann zeta function $\zeta_S(s) = \prod_{p \in S}(1-p^{-s})^{-1}$. We note that $\zeta_S$ agrees with the Dedekind zeta function of the ring $\bbZ_{S'} = \bbZ[\frac{1}{p} \mid p \not\in S]$. 

Conjecture \ref{conjA} is prompted by number theoretic considerations. With a bit of elementary number theory one can show (see Lemma \ref{lem:basic-lemma} below)
\[
	\frac{1}{L}(\alpha_D(S)+1) \leq \alpha(G_S) \leq \alpha(H_S) \leq \alpha_D(S)+1.
\]
where $L$ is a Linnik constant, i.e., a constant $L$ such that the least prime $p \equiv b \mod m$ satisfies $p \ll m^L$ for all $m,b$ with $(m,b)=1$.
Linnik showed that such a constant exists. Currently $L=5$ is the best known value \cite{Xylouris}.
Chowla conjectured that $L = 1+\epsilon$ is a Linnik constant for every $\epsilon > 0$; see \cite{Chowla1934}. Hence Chowla's conjecture implies Conjecture \ref{conjA}. However, currently Chowla's conjecture seems to be out of reach and (as Chowla noticed) even the extended Riemann hypothesis only gives $L = 2+\epsilon$ for every $\epsilon > 0$.

\medskip

Conjecture \ref{conjA} has interesting consequences. It implies that the Weil abscissa 
of a procyclic group $G$ only depends on the set $S$ of primes dividing the order of $G$. Another consequence is:
\begin{conjecture}\label{conjVal}
Every real number $1\leq \beta \leq 2$ is the abscissa of convergence of a procyclic group.
\end{conjecture}

The purpose of this paper is to prove Conjecture \ref{conjA} almost surely for randomly chosen sets $S$ of primes. This will enable us to deduce the validity of Conjecture \ref{conjVal}.

\medskip

Let \( \delta \in (0,1) \) be given. We say that a set $S$ of prime numbers is \emph{random w.r.t. the $\delta$-model}, if the events $p \in S$ and $p' \in S$ are independent for $p \neq p'$ and occur with probability 
\[ \bfP(p\in S) = \frac{p^{1-\delta}-1}{p-1}. \]

\begin{thm}\label{thm:main-1}
Let $S$ be a random set of primes w.r.t.\ the $\delta$-model. Then almost surely
\[
	\alpha(G_S) = \alpha(H_S) = \alpha_D(S)+1 = 2-\delta.
\] 
\end{thm}
The key step in the proof of Theorem \ref{thm:main-1} is to bound the variance of a family of random variables using the Bombieri-Vinogradov theorem. 
\begin{corollary}\label{cor:conjVal}
Conjecture \ref{conjVal} holds true.
\end{corollary}
The random model used here can be perturbed slightly without changing the result.
It should be noted that $G_S$ and $H_S$ are the minimal and maximal procyclic groups such that $S$ is the set of primes dividing the order. This shows that for every ``random'' procyclic group $G$ such that the primes dividing the order of $G$ are distributed in the $\delta$-model, the Weil abscissa satisfies $\alpha(G) = 2- \delta$ almost surely.

\smallskip

It seems desirable to understand in detail for which sets $S$ the Weil abscissae of $G_S$ and $H_S$ are maximal, i.e. take the value $2$. In other words, how ``small'' can a set $S$ with $\alpha(G_S) = 2$ be? 
Modifying the proof of a result of Mirsky \cite{Mirsky}, we derive the following theorem.

\begin{thm}\label{thm:main-2}
Let $S$ be a thick set of primes. Then $\alpha(G_S) = \alpha(H_S) = 2$.
\end{thm}
Here a set $S$ is primes is \emph{thin}, if $\alpha_D(S) < 1$. The complement of a thin set is called \emph{thick}.

\medskip
We note that Corollary \ref{cor:conjVal} is a mere existence result. The following problem remains open.
\begin{prob}
Give an explicit description of a procyclic group $G$ with $1 < \alpha(G) <2$.
\end{prob}

The preliminary observation, that served as a starting point for this investigation, will be discussed in Section~\ref{sec:prelim}. Section~\ref{sec:mirsky} elaborates on a result of Mirsky and contains the proof of Theorem~\ref{thm:main-2}. The last section contains the proof of our main theorem.

\section{Preliminaries}\label{sec:prelim}
Let $S \subseteq \mathcal{P}$ be a set of prime numbers. Throughout $S' = \mathcal{P} \setminus S$ denotes the complement of $S$ in the set of all prime numbers. A natural number is an $S$-number (resp. $S'$-number), if all of its prime factors lie in $S$ (resp.\ in $S'$). We write $\langle S \rangle$ to denote the set of all $S$-numbers. For $n \in \bbN$, we denote the largest $S$-number dividing $n$ by $n_S$.  In particular, we have $n = n_S n_{S'}$. We write $\rad(n)$ to denote the radical of $n$ and we put $\rad_S(n) := \rad(n)_S$.
The Dedekind zeta function of the ring $\bbZ_{S'} = \mathbb{Z}[{S'}^{-1}]$ is the partial Riemann zeta function
\[
	\zeta_{S}(s) = \sum_{m \in \langle S \rangle} m^{-s} = \prod_{p \in S} (1-p^{-s})^{-1}.
\]
Recall that the abscissa of convergence of $\zeta_S$ is $\alpha_D(S)$; clearly $ \alpha_D(S) \leq 1$.

\begin{lemma}\label{lem:basic-lemma}
Let $H_S = \widehat{\bbZ_{S'}} = \prod_{p \in S} \bbZ_p$ and let $G_S = \prod_{p\in S} C_p$.
If $S \neq \emptyset$, then the Weil abscissae satisfy
\[
	\frac{1}{L}(\alpha_D(S)+1) \leq \alpha(G_S) \leq \alpha(H_S) \leq \alpha_D(S)+1
\]
where $L$ denotes a Linnik constant.
\end{lemma}
\begin{proof}
As $G_S$ is a quotient of $H_S$ the inequality $\alpha(G_S)\leq \alpha(H_S)$ is clear (cf.~\cite[Lemma 4.1]{CCKV24}).

Consider the upper bound first. We observe that $|\Hom(H_S,\bbF_{p^j})| = (p^j-1)_S$. Let $\Lambda$ denote the von Mangoldt function.
We define $\Lambda'(n) =  \frac{\Lambda(n)}{\log(n)}$ for $n \geq 2$ and we put
$\Lambda'(1) = 0$. We observe that $\Lambda'(n) \leq 1$ for all~$n$. For all real $s > \alpha_D(S)+1$ we have
\begin{align*}
\log \zeta^{\text{W}}_{G_S}(s) &= \sum_{n \geq 2} (n-1)_S \ \Lambda'(n)n^{-s} \leq \sum_{n\geq 2} (n-1)_S \ n^{-s} \\
 &\leq \sum_{m\geq 1} m_S \ m^{-s} 
= \sum_{k \in \langle S\rangle, \ell \in \langle S' \rangle} k (k\ell)^{-s}\\
&= \zeta_{S}(s-1) \zeta_{S'}(s).
\end{align*} 
Since $S \neq \emptyset$, we have $\alpha_D(S) \geq 0$ and hence 
$s > \alpha_D(S')+1 \geq 1$; in particular, the series converges.

For the lower bound we observe that $|\Hom(G_S,\bbF_{p^j})| = \rad_S(p^j-1)$. 
Define $\mu_2(m) = 1$ if $m$ is square-free and $\mu_2(m) =0$ otherwise. 
The Dirichlet series 
\[
	\sum_{m \in \bbN_{S}} \mu_2(m) m^{-s} = \prod_{p \in S} (1+p^{-s}) = \frac{\zeta_{S}(s)}{\zeta_{S}(2s)}
\]
has the same absicssa of convergence as $\zeta_{S}$.
Since $L$ is a Linnik constant, there is some $C >0 $ such that
for all $m \in \langle S \rangle$, we can find a prime number $p(m)$ with $p(m) \equiv 1 \bmod m$ and $p(m) \leq Cm^L$.
Every prime $p$ equals $p(m)$ for at most $d(p-1)$ different numbers $m$, where $d$ denotes the divisor function.
Let $\epsilon > 0$ be given.
The well-known divisor bound asserts that $d(p-1) = O_\epsilon(p^\epsilon)$ and this
allows to conclude
\begin{align*}
	\log(\zeta^{\text{W}}_{G_S})(s) &\geq \sum_{p} \rad_S(p-1) p^{-s}\\
	 &\geq \sum_{m \in \langle S\rangle} \mu_2(m) 
	\rad_S(p(m)-1)d(p(m))^{-1} p(m)^{-s}\\ 
	&\gg  \sum_{m \in \langle S\rangle} \mu_2(m) m  p(m)^{-s-\epsilon} \\
	&\gg \sum_{m \in \langle S\rangle} \mu_2(m) m^{1-L(s+\epsilon)}.
\end{align*}
This Dirichlet series diverges for $L(s+\epsilon) < \alpha_D(\bbZ_{S'})+1$. Since $\epsilon$ is arbitrary this implies the assertion.
\end{proof}

\section{A variation of a result of Mirsky}\label{sec:mirsky}
Mirsky \cite{Mirsky} proved that the set of prime numbers $p$ such that $p-1$ is $k$-free has a positive natural density
$\prod_{p} \bigl(1- \frac{1}{(p-1)p^{k-1}}\bigr)$.
Here we need a slight variation of his result.
\begin{definition}
Let $f \colon \mathcal{P}\to \mathbb{N}_0$ be a function. We say that a natural number $n$ is \emph{$f$-bounded}, if
\[
	\nu_p(n) \leq f(n)
\]
for all primes $p$. Here $\nu_p(n) = \max \{j\in\bbN_0 \mid \ p^j | n\}$.
\end{definition}
If $f$ is the constant function with value $k-1$, then $f$-bounded numbers are exactly the $k$-free numbers.
\begin{definition}
Recall that a set $S$ is primes is \emph{thin}, if $\alpha_D(S) < 1$.
A set of primes is \emph{thick}, if its complement is thin.
A function $f\colon \mathcal{P} \to \bbN_0$ is \emph{thick}, if its support is thick; i.e., $f^{-1}(0)$ is thin.
\end{definition}
Following closely along the lines of Mirsky's approach, one obtains:
\begin{thm}\label{thm:Mirsky-var}
Let $f\colon \mathcal{P} \to \mathbb{N}_0$ be thick and let $A > 0$ be arbitrary.
Let $U_f(x)$ denote the set of primes $p \leq x$ such that $p-1$ is $f$-bounded.
Then
\[
	|U_f(x)| = \prod_{p} \left(1- \frac{1}{(p-1)p^{f(p)}}\right) \mathrm{Li}(x) + O\Bigl(\frac{x}{\log^A(x)}\Bigr)
\]
as $x$ tends to infinity.
\end{thm}
The thickness of $f$ implies that $\lambda_f = \prod_{p} \left(1- \frac{1}{(p-1)p^{f(p)}}\right)$ converges. For $f(2) > 0$ the limit is nonzero. Since Mirsky omitted the proof of his theorem in \cite{Mirsky}, we decided to include a proof.
\begin{proof}
Let $p$ be a prime, we write $p^{(f)}:= p^{f(p)}$ and extend this multiplicatively to all natural numbers, i.e. for $n = \prod_{i=1}^r p_i^{e_i}$ we have $n^{(f)} = \prod_{i=1}^r p_i^{f(p_i)e_i}$.
For convenience we write $n^{(f)+1}:= n\cdot n^{(f)}$ and we write $a \mid_f n$ if $a^{(f)+1} \mid n$.

The M\"obius function $\mu$ gives rise to
\[
	\mu_f(n) := \sum_{a \mid_f n} \mu(a) = \begin{cases} 1 & \text{ if  $a$ is $f$-bounded }\\
	0 & \text{ otherwise }\end{cases}.
\]
This gives
\[
	|U_f(x)| = \sum_{p \in U_f(x)} 1 = \sum_{p \leq x} \mu_f(p-1) =\sum_{p \leq x} \sum_{a \mid_f p-1} \mu(a) = \sum_{a \leq x} \mu(a) \pi(x,a^{(f)+1},1)
\]
where $\pi(x,m,b)$ denotes the number of primes $p \leq x$ with $p \equiv b \bmod m$.

Let $B > 1$ be arbitrary. Observing that $\lambda_f = \sum_{a} \frac{\mu(a)}{\phi(a^{(f)+1})}$ (where $\phi$ is Euler's totient function) we may write $|U_f(x)| - \lambda_f \Li(x)$ as a sum
\begin{align}
 \label{first:term}\sum_{a \leq \log^B(x)} &\mu(a)\Bigl(\pi(x,a^{(f)+1},1) - \frac{\Li(x)}{\phi(a^{(f)+1})}\Bigr)\\
  \label{second:term}  &+\sum_{\log^B(x) < a \leq x} \mu(a) \pi(x,a^{(f)+1},1)\\ 
 \label{third:term} &-\sum_{a >\log^B(x)} \mu(a)\frac{\Li(x)}{\phi(a^{(f)+1})} 
\end{align}
The summand \eqref{first:term} is $O\bigl(\frac{x}{\log^{A}(x)}\bigr)$ by the Theorem of Siegel-Walfisz \cite[Cor.~5.29]{IwaniecKowalski}. To approach the other two terms we use the assumption that $f$ is thick. Let $S = f^{-1}(0)$ be the thin zero set of $f$. By assumption $\alpha_D(S) < 1$, so there is $c> 0$ with $c+\alpha_D(S) < 1$; this entails
 \begin{equation}
	\sum_{\substack{a \in \langle S \rangle\\ a > y}} \frac{1}{a} \leq y^{-c} \sum_{\substack{a \in \langle S \rangle\\ a > y}} \frac{1}{a^{1-c}} \ll y^{-c}
\end{equation}
as $y$ tends to infinity.

Now consider the summand \eqref{second:term}. The rough estimate $\pi(x,n,1) \leq \frac{x}{n}$ suffices to deduce
\begin{align*}
 \Bigl| \sum_{\log^B(x) < a \leq x} \mu(a) \pi(x,a^{(f)+1},1) \Bigr| &\leq \sum_{\log^B(x) < a } \frac{x}{a^{(f)+1}}\\
  &\leq x \Bigl(\sum_{\substack{a \in \langle S \rangle\\ a > \log^{B/2}(x)}} \frac{1}{a}\sum_{b \in \langle S' \rangle} \frac{1}{b^2} + \sum_{a \in \langle S \rangle} \frac{1}{a} \sum_{\substack{b \in \langle S' \rangle\\ b > \log^{B/2}(x)}} \frac{1}{b^2}\Bigr)\\
  &\ll x \Bigl(\frac{1}{\log^{cB/2}(x)} + \frac{1}{\log^{B/2}(x)}\Bigr).
\end{align*}
Recall that $B$ was arbitrary; this proves that \eqref{second:term} is $O\bigl(\frac{x}{\log^{A}(x)}\bigr)$.
The last summand \eqref{third:term} can be bounded analogously using that $n^{1-\epsilon} \ll \phi(n)$ for all $\epsilon > 0$
(see \cite[Thm.~327]{HardyWright}).  
\end{proof}

With this preparation we are able to prove Theorem~\ref{thm:main-2}
\begin{proof}[Proof of Theorem~\ref{thm:main-2}]
It is known that $\alpha(G_S) \leq \alpha(H_S) \leq \alpha(\widehat{\bbZ}) = 2$ by Lemma \ref{lem:basic-lemma} and \cite[\S 5.1]{CCKV24}. 
It suffices to prove $\alpha(G_S) = 2$. We define a function $f \colon \mathcal{P} \to \bbN_0$ with support $S \cup \{2\}$. For odd primes $p$ we define $f(p) = 1$ iff $p\in S$ and we impose $f(2) = 1$ to ensure that $\lambda_f \neq 0$. It follows from Theorem \ref{thm:Mirsky-var} via partial summation that
\[
	\sum_{p\in U_f(x)} \frac{1}{p} = \lambda_f \log\log(x) + O(1).
\]
For $p\in U_f(x)$ we have $\rad_S(p-1) = \begin{cases} p-1 & \text{ if } 2 \in S\\
\frac{p-1}{2} & \text{ if } 2 \not\in S \end{cases}.$ This allows us to deduce
\begin{align*}
  \sum_{p^{j} \leq x} \frac{\rad_S(p^j-1)}{j}p^{-2j} &\geq \sum_{p \in U_f(x)} \frac{(p-1)}{2} p^{-2}\\
   &\gg \sum_{p \in U_f(x)} p^{-1}  = \lambda_f \log\log(x) + O(1).
  \end{align*} 
 We conclude that the Weil representation zeta function $\zeta^W_{G_S}(s)$ diverges at $s=2$.
\end{proof}

\section{Random procyclic groups}
Let \( \delta \in (0,1) \) be given. Recall that a set $S$ of prime numbers is \emph{random w.r.t. the $\delta$-model}, if the events $p \in S$ and $p' \in S$ are independent for $p \neq p'$ and occur with probability 
\[ \bfP(p\in S) = \frac{p^{1-\delta}-1}{p-1}. \]
More precisely, we encode sets $S$ of primes as characteristic functions; this leads to the probability space \(\Omega = \prod_{p \in \mathcal{P}} \{0,1\}\) 
with the product probability measure $\prod_{p} \mu_{p,\delta}$ where $\mu_{p,\delta}(1) = \frac{p^{1-\delta}-1}{p-1}$.
From now on we consider the procyclic groups $G_S$ and $H_S$ for $S$ chosen randomly in the $\delta$-model.

\begin{lemma}\label{lem:limsup}
Let $(X_t)_{t \in \bbR}$ be a one-parameter family of non-negative random variables with $0 < \bfE(X_t) < \infty$ .
Assume that 
\[
	\bfVar(X_t) = O(\bfE(X_t)^2) \quad \text{ as } t \to \infty.
\]
Then $\bfP(\limsup_{t \to \infty} \frac{X_t}{\bfE(X_t)} = 0 ) < 1$. 
\end{lemma}
\begin{proof}
Write $\sigma^2_t = \bfVar(X_t)$.
Assume that $\sigma^2_t \leq \lambda \bfE(X_t)^2$ for $\lambda > 0$ and all $t \geq n_0$.
The Cantelli-Chebyshev inequality allows us to deduce for all $t \geq n_0$
\begin{align*}
	\bfP\Bigl( X_t < \frac{1}{2}\bfE(X_t)\Bigr) &= \bfP\Bigl( X_t - \bfE(X_t) < -\frac{1}{2}\bfE(X_t)\Bigr)  \leq \frac{\sigma^2_t}{\sigma^2_t+ \frac{1}{4}\bfE(X_t)^2}\\
	&\leq \frac{\sigma^2_t}{\sigma^2_t + \frac{1}{4\lambda}\sigma^2_t} = \frac{4\lambda}{4\lambda+1}.
\end{align*}
Assume for a contradiction that $\bfP(\limsup_{t \to \infty} \frac{X_t}{\bfE(X_t)} = 0 ) =1$.
For almost all $\omega \in \Omega$, there is $t_\omega \in \bbR_{>0}$ such that $X_t(\omega) < \frac{1}{2}\bfE(X_t)$ for all $t \geq t_\omega$.  
In particular, the union of the events $A_n = \bigl\{ \omega \in \Omega \mid  X_t(\omega) < \frac{1}{2}\bfE(X_t) \text{ for all } t \geq n\bigr\}$ occurs with probability~$1$. This implies that 
there is $n \geq n_0$ with 
\[ \frac{4\lambda}{4\lambda+1} \geq \bfP\Bigl(X_n < \frac{1}{2}\bfE(X_n)\Bigr) \geq \bfP(A_n) > \frac{4\lambda}{4\lambda+1}\]
and this gives the desired contradiction.
\end{proof}

The following number theoretic result is the key ingredient in our main theorem; it is a consequence of the Bombieri-Vinogradov theorem.
\begin{lemma}\label{lem:bombieri-vinogradov}
Let $\delta\in (0,1)$ and let $A > 2$. Then 
\[
	\sum_{a \leq t} a^\delta \pi(t,a,1)^2  = b_\delta \Li(t)^2 + O\Bigl(\frac{t^2}{\log^{A}(t)}\Bigr)
\]
 as $t \to \infty$ where $b_\delta = \sum_{a=1}^\infty \frac{a^\delta}{\phi(a)^2}$. The implied constants depend on $A$ and $\delta$.
 \end{lemma}
\begin{proof}
It is well-known that $a^{1-\epsilon} \ll \phi(a)$ for every $\epsilon > 0$ (see \cite[Thm.~327]{HardyWright}), hence the series $\sum_{a}\frac{a^\delta}{\phi(a)^2}$ converges.
Pick $0 < \theta < \frac{1}{2}$. We write
\[
	\sum_{a \leq t} a^\delta \pi(t,a,1)^2 = \underbrace{\sum_{a \leq t^\theta}a^\delta \pi(t,a,1)^2}_{\Sigma_1} +  \underbrace{\sum_{t^\theta< a \leq t}a^\delta \pi(t,a,1)^2}_{\Sigma_2}.
\]
To bound $\Sigma_2$ we use the observation $\pi(t,a,1) \leq \frac{t}{a}$ and deduce
\[
	\Sigma_2 \leq \sum_{t^\theta< a \leq t}a^\delta \frac{t^2}{a^2} = t^2 \sum_{t^\theta< a \leq t} a^{\delta- 2} \ll t^2\int_{t^\theta}^tx^{\delta-2} \mathrm{d}x \ll t^{2+(\delta-1)\theta} \ll \frac{t^2}{\log^{A}(t)},
\]
since $\delta < 1$.

We rewrite $\Sigma_1-b_\delta\Li(t)^2$ as
\begin{equation}\label{eq:decomposeS1}
	\sum_{a \leq t^\theta} a^\delta \Bigl(\pi(t,a,1)-\frac{\Li(t)}{\phi(a)}\Bigr)^2 + 2\Li(t)\sum_{a \leq t^\theta}\frac{a^\delta}{\phi(a)} \Bigl(\pi(t,a,1)-\frac{\Li(t)}{\phi(a)}\Bigr) - \Li(t)^2\sum_{a > t^\theta} \frac{a^\delta}{\phi(a)^2}.
\end{equation}
Using $a^{1-\epsilon} \ll \phi(a)$ for every $\epsilon > 0$, we see that
the last summand in \eqref{eq:decomposeS1} is bounded by
\[
	\Li(t)^2\sum_{a > t^\theta} \frac{a^\delta}{\phi(a)^2}
	 \ll \Li(t)^2 \int_{t^\theta}^\infty x^{-2+\delta+2\epsilon} \mathrm{d}x
	  \ll \frac{t^{1+\delta+2\epsilon}}{\log^{2}(t)}
\]
and for $\delta+2\epsilon < 1$, the last expression is $O\bigl(\frac{t^2}{\log^{A}(t)}\bigr)$.
The other two summands in \eqref{eq:decomposeS1} can be treated using the Bombieri-Vinogradov theorem (see e.g.~\cite[Theorem~9.2.1]{SieveMethods}). It readily implies
\[	
	\sum_{a \leq t^\theta}  \Bigl|\pi(t,a,1) - \frac{\Li(t)}{\phi(a)}\Bigr|  \ll \frac{t}{\log^{A}(t)}.
\]
and this allows us to deduce (using $a^{\delta} \ll \phi(a)$)
\[	
	\sum_{a \leq t^\theta}  \frac{a^\delta}{\phi(a)}\Bigl|\pi(t,a,1) - \frac{\Li(t)}{\phi(a)}\Bigr|  \ll \frac{t}{\log^{A}(t)}.
\]
The theorem of Bombieri-Vinogradov as stated in \cite{ElliottHalberstam} implies
\begin{equation*}
	\sum_{a \leq t^\theta} \phi(a) \Bigl(\pi(t,a,1)-\frac{\Li(t)}{\phi(a)}\Bigr)^2 = O\Bigl(\frac{t^2}{\log^{A}(t)}\Bigr).
\end{equation*}
As a consequence (using $a^{\delta} \ll \phi(a)$) this implies
\begin{align*}
	\sum_{a \leq t^\theta} a^\delta \Bigl(\pi(t,a,1)-\frac{\Li(t)}{\phi(a)}\Bigr)^2
	&\ll\sum_{a \leq t^\theta} \phi(a) \Bigl(\pi(t,a,1)-\frac{\Li(t)}{\phi(a)}\Bigr)^2\\ 
	&\ll \frac{t^2}{\log^{A}(t)}
\end{align*}
as $t \to \infty$.
This completes the proof.
\end{proof}

\begin{proof}[Proof of Theorem \ref{thm:main-1}]
For each $n \in \bbN$ we define the random variable
\[
	I_n = \begin{cases} 1 & \text{ if } n \in \langle S \rangle \\
	0 & \text{ otherwise } \end{cases}.
\]
The expectation of $I_n$ is the probability that $n$ is an $S$-number, i.e. that all prime factors of $n$ belong to $S$:
\[
	\bfE(I_n) = \prod_{p \mid n}\frac{p^{1-\delta}-1}{p-1} = \rad(n)^{-\delta} \prod_{p\mid n} \frac{p-p^{\delta}}{p-1} \leq \rad(n)^{-\delta} .
\]
For $s \in \bbR_{>0}$ we define a random variable
\[
	Y(s) = \sum_{m\in \bbN} I_m m^{-s}.
\]
Since the random variables $I_m$ are positive, the monotone convergence theorem implies that
\[
	\bfE(Y(s)) = \sum_{m \in \bbN} \bfE(I_m) m^{-s} \leq \sum_{m \in \bbN} \rad(m)^{-\delta} m^{-s} 
	=\prod_{p} \bigl(1+\frac{p^{-\delta-s}}{1-p^{-s}}\bigr).
\]
The Dirichlet series on the right hand side converges absolutely for all $s > 1-\delta$.
If the expectation of a random variable is finite, the almost surely the random variable is finite. With Lemma \ref{lem:basic-lemma} this gives
$\alpha(G_S) \leq \alpha(H_S) \leq 2- \delta$ almost surely.

\medskip

For the lower bound, recall that $\log\zeta^W_{G_S}(s) = \sum_{p\in \mathcal{P}}\sum_{j=1}^\infty \rad_S(p^j-1)p^{-sj}$. We consider the partial sums of this series for $s= 1-\delta$ as a random variable in $S$. The $S$-radical of a number $n$ gives rise to the random variable $R_n = \prod_{\substack{p \mid n\\p \in S}} p = \prod_{p \mid n} \bigl((p-1)I_p+1\bigr)$.
As the $I_p$ are independent, we have 
\[ \bfE(R_n) = \prod_{p \mid n} \Bigl((p-1)\frac{p^{1-\delta}-1}{p-1}+1\Bigr) = \rad(n)^{1-\delta}. \]
For every $t \in \bbR_{>0}$ we consider the random variable
\[
	X_t = \sum_{p \leq t} R_{p-1} p^{\delta-1}.
\]
We are going to verify the conditions of Lemma \ref{lem:limsup} for the family $X_t$.

As a first step we show that the expectation of $X_t$ is of the order $\Li(t)$, i.e. $\bfE(X_t) \asymp \Li(t)$.
For the upper bound, we use the prime number theorem. For every $A > 1$ we have
\[
	\bfE(X_t) = \sum_{p \leq t} \rad(p-1)^{1-\delta} p^{\delta-1} \leq \sum_{p \leq t} 1 = \Li(t) + O\Bigl(\frac{t}{\log^{A}(t)}\Bigr)
\]
The lower bound is a consequence of Mirsky's Theorem; i.e., we count primes $p \leq t$ such that $p-1$ is squarefree (apply Theorem \ref{thm:Mirsky-var} with $f$ the constant $1$-function). This gives
\begin{align*}
	\bfE(X_t) &= \sum_{p \leq t} \rad(p-1)^{1-\delta} p^{\delta-1} \geq  \sum_{\substack{p \leq t\\p-1 \text{ sqf.}}} \left(\frac{p-1}{p}\right)^{1-\delta}\\
	&\geq 2^{-1+\delta} \sum_{\substack{p \leq t\\p-1 \text{ sqf.}}} 1 = 2^{-1+\delta}U_f(t) = c_\delta\Li(x) + O\Bigl(\frac{t}{\log^{A}(t)}\Bigr)
\end{align*}
for some $c_\delta > 0$.

We proceed to estimate the variance of $X_t$. As $X_t$ is bounded, we have
	$\bfVar(X_t) = \bfE(X_t^2)-\bfE(X_t)^2$
and to apply Lemma \ref{lem:limsup} it suffices to show that $\bfE(X_t^2) = O(\bfE(X_t)^2) = O(\Li(t)^2)$ as $t \to \infty$.

Let $n,m \in \bbN$ be integers. The expectation of $R_nR_m$ is
\[
	\prod_{\substack{p \mid n\\p\nmid m}}p^{1-\delta}\prod_{\substack{p \mid m\\p\nmid n}} p^{1-\delta}
	\prod_{p \mid (n,m)}\bigl(p^{2-\delta}(1-p^{\delta-1}+p^{-1})\bigr) \leq 2\rad(n)^{1-\delta}\rad(m)^{1-\delta}\rad(n,m)^\delta.
\]
Here $(n,m)$ denotes the gcd of $n$ and $m$ and $\rad(n,m)$ is the radical of $(n,m)$. 
This gives, using Lemma~\ref{lem:bombieri-vinogradov} in the last step,
\begin{align*}
	\bfE(X_t^2) &= \sum_{p_1,p_2\leq t} \bfE(R_{p_1-1}R_{p_2-1}) (p_1p_2)^{\delta-1}\\
	&\leq 2\sum_{p_1,p_2\leq t} \rad(p_1-1)^{1-\delta}\rad(p_2-1)^{1-\delta}\rad(p_1-1,p_2-1)^\delta (p_1p_2)^{\delta-1}\\
	&\leq 2\sum_{a \leq t} a^\delta\sum_{p_1,p_2 \equiv 1 \bmod a}\rad(p_1-1)^{1-\delta}\rad(p_2-1)^{1-\delta} (p_1p_2)^{\delta-1}\\
	&\leq 2\sum_{a \leq t} a^\delta \pi(t,a,1)^2 \ll \Li(t)^2
\end{align*}
We conclude using Lemma~\ref{lem:limsup} that
$\bfP\left(\limsup_{t\to\infty} \frac{X_t}{\Li(t)} > 0\right) > 0$.
Positivity of $\limsup_{t\to\infty} \frac{X_t}{\Li(t)}$ does not depend on any finite set of primes in $S$, hence it is a tail event and Kolmogorov's 0-1 law implies $\bfP\left(\limsup_{t\to\infty} \frac{X_t}{\Li(t)} > 0\right) = 1$.

This allows us to conclude the proof. Let $S$ be a set of primes with $\limsup_{t\to \infty} \frac{X_t(S)}{\Li(t)} > 0$. Then there is $c >0$ and an sequence $t_1 < t_2 <\dots$ tending to infinity, such that
\[\sum_{p\leq t_j} \rad_S(p-1) p^{\delta-1} > c \Li(t_j)\]
for all $j$.
Let $\epsilon > 0$ be arbitrary. Clearly
\[
	\sum_{p\leq t_j} \rad_S(p-1) p^{\delta-2+\epsilon} \geq c\Li(t_j) t_j^{-1+\epsilon}
\]
for all $j$ and the right hand side tends to infinity with $t_j$. Hence $\alpha(G_S) > 2-\delta-\epsilon$ almost surely for all $\epsilon > 0$.
\end{proof}

\end{document}